\documentclass{article}
\usepackage{amsmath}
\usepackage{amssymb}
\usepackage{amsthm}
\usepackage{mathtools}

\DeclarePairedDelimiter\floor{\lfloor}{\rfloor}
\usepackage[utf8]{inputenc}
\usepackage[english]{babel}
\usepackage{color}
\usepackage[a4paper,margin=1in,nohead]{geometry}

\newcommand{\RR}{\mathbb{R}}

\newcommand{\pts}{\mathcal P}
\newcommand{\spheres}{\mathcal S}
\newcommand{\circs}{\Gamma}
\newcommand{\CaC}{\mathcal C}
\newcommand{\Mod}[1]{\ (\mathrm{mod}\ #1)}
\def\eps{{\varepsilon}}

\newcommand{\ignore}[1]{}

\newcommand{\parag}[1]{\vspace{2mm}

\noindent{\bf #1} }

\newtheorem{theorem}{Theorem}[section]

\newtheorem{proposition}[theorem]{Proposition}

\newtheorem{lemma}[theorem]{Lemma}

\title{On the Number of Discrete Chains}
\date{\today}
\author{
Eyvindur Ari Palsson\thanks{Department of Mathematics, Virginia Tech, Blacksburg, VA, USA. {\sl palsson@vt.edu}. Supported by Simons Foundation Grant \#360560.}
\and
Steven Senger\thanks{Department of Mathematics, Missouri State University, Springfield, MO, USA. {\sl stevensenger@missouristate.edu}}
\and
Adam Sheffer\thanks{Department of Mathematics, Baruch College, City University of New York, NY, USA.
{\sl adamsh@gmail.com}. Supported by NSF award DMS-1710305.}}

\begin{document}
\maketitle
\begin{abstract}
We study a generalization of Erd\H os's unit distances problem to chains of $k$ distances.
Given $\pts$, a set of $n$ points, and a sequence of distances $(\delta_1,\ldots,\delta_k)$, we study the maximum possible number of tuples of distinct points $(p_1,\ldots,p_{k+1})\in \pts^{k+1}$ satisfying $|p_j p_{j+1}|=\delta_j$ for every $1\le j \le k$.
We study the problem in $\RR^2$ and in $\RR^3$, and derive upper and lower bounds for this family of problems.
\end{abstract}

\section{Introduction}

\emph{Erd\H os's} unit distances problem is one of the main open problems of Discrete Geometry.
To quote the book \emph{Research Problems in Discrete Geometry} \cite{BMP05}, this is ``possibly the best known (and
simplest to explain) problem in combinatorial geometry''.
The problem simply asks: In a set of $n$ points in $\RR^2$, what is the maximum number of pairs of point at a distance of one?
We denote this maximum value as $u(n)$.
In 1946, Erd\H os \cite{Erd46} constructed a configuration of $n$ points that span $n^{c\sqrt{\log n}}$ unit distances, for some constant $c$.
While over 70 years have passed, this remains the current best lower bound for $u(n)$.
In 1984, Spencer, Szemer\'edi, and Trotter \cite{SST84} derived the current best upper bound $u(n)=O(n^{4/3})$.

Although the unit distances problem is a central open problem, with many people studying it and with connections to many other problems, no new bounds has been obtained for $u(n)$ since 1984.
As is often the case in such situations of stagnation, researchers began studying variants of the problem.
For example, Matou\v sek \cite{mat11} showed that when replacing the Euclidean distance norm with a generic norm, the maximum number of unit distances becomes significantly smaller.
Thus, to solve the unit distances problem, one probably has to rely on a property that is special to the Euclidean norm.
For other variants of the problem, see for example \cite{BMP05}.

In this note, we study the following generalization of the unit distances problem.
We have a set $\pts$ of $n$ points in $\RR^2$, a positive integer $k$, and sequence of distances $(\delta_1,\delta_2,\ldots,\delta_{k})$.
Denote the distance between two points $p,q\in \RR^2$ as $|pq|$.
We define a $k$-\emph{chain} to be a $(k+1)$-tuple of distinct points $(p_1,\ldots,p_{k+1})\in \pts^{k+1}$ such that for every $1\le j \le k$ we have $|p_jp_{j+1}|=\delta_j$.
We are interested in the maximum possible value of such chains.
By considering the case of $k=1$, we note that the chains problem is indeed a generalization of the unit distances problem.
A continuous $d$-dimensional variant of this problem was previously studied in \cite{BIT16}.

\parag{Our results.}
For a positive integer $k$, we denote as $\CaC_{k}(n)$ the maximum number of $k$-chains that can be spanned by $n$ points in $\RR^2$, for any sequence of distances $(\delta_1,\delta_2,\ldots,\delta_{k})$.
We only consider the case where $k$ is a constant that does not depend on $n$.
We trivially have $\CaC_{0}(n)=n$, and we also have $\CaC_{1}(n)=u(n)=O(n^{4/3})$.
While obtaining a tight bound for $\CaC_{1}(n)$ is equivalent to a notoriously difficult open problem, it is surprisingly easy to show that $\CaC_{2}(n)=\Theta(n^2)$.
See Lemma \ref{le:hinges} below.

The following theorem is our main result in $\RR^2$.

\begin{theorem}\label{th:mainR2}
For every integer $k\ge 3$, we have
\[ \CaC_{k}(n)= O\left( n^{2k/5+ 1+\gamma(k)}\right), \]
where
\[
\gamma(k) =
\begin{cases}
\frac{1}{75}\cdot\left( 4 - 4\cdot (-1/4)^{k/4}  \right) \qquad & \text{ if } k \equiv 0 \Mod{4}, \\
\frac{1}{75}\cdot\left( 4 - 9\cdot (-1/4)^{\floor{k/4}}  \right) & \text{ if } k \equiv 1 \Mod{4}, \\
\frac{1}{75}\cdot\left( 4 + 11\cdot (-1/4)^{\floor{k/4}} \right) & \text{ if } k \equiv 2 \Mod{4}, \\
\frac{1}{75}\cdot\left( 4 - \frac{13}{2} \cdot (-1/4)^{\floor{k/4}}   \right) & \text{ if } k \equiv 3 \Mod{4}.
\end{cases}
\]
\end{theorem}

When $k\ge 3$ we have that $\gamma(k) \leq 1/12$.
As $k\to \infty$ we have that $\gamma(k) \to \frac{4}{75}$.

In the other direction, we derive the following lower bounds for $\CaC_k(n)$. This construction is due to Lauren Childs.

\begin{proposition}[Childs' construction]\label{pr:LowerR2}
For any integer $k\ge 0$, we have
\[
\CaC_k(n) =
\begin{cases}
\Omega\left(n^{k/3+1}\right) \qquad & \text{if } \, k\equiv 0 \Mod{3}, \\
\Omega\left(n^{(k+2)/3}\right) & \text{if } \, k\equiv 1 \Mod{3}, \\
\Omega\left(n^{(k+1)/3+1}\right) & \text{if } \, k\equiv 2 \Mod{3}.
\end{cases}
\]
\end{proposition}

Note that there is still a polynomial gap between the bounds of Theorem \ref{th:mainR2} and Proposition \ref{pr:LowerR2}.
We believe that the lower bounds are tight up to subpolynomial terms.
Indeed, when assuming the unit distances conjecture $u(n) = \Theta\left(n^{c\sqrt{\log n}}\right)$, one can show that the bounds of Proposition \ref{pr:LowerR2} are tight up to subpolynomial factors.

\begin{proposition}\label{pr:OptimisticBound}
For every $k\ge 3$, we have
\[
\CaC_{k}(n) =
\begin{cases}
O\left(n \cdot u_2(n)^{k/3}\right) \qquad & \text{if } \, k\equiv 0 \Mod{3}, \\
O\left(u_2(n)^{(k+2)/3}\right) & \text{if } \, k\equiv 1 \Mod{3}, \\
O\left(n^2\cdot u_2(n)^{(k-2)/3}\right) & \text{if } \, k\equiv 2 \Mod{3}.
\end{cases}
\]
\end{proposition}


\parag{Chains in three dimensions.}
The unit distances problem has also been studied in $\RR^3$ for many decades.
Let $u_3(n)$ denote the maximum number of unit distances that can be spanned by $n$ points in $\RR^3$.
In 1960 Erd\H os \cite{erd60} derived the bound $u_3(n) =\Omega(n^{4/3}\log \log n)$, and this remains the current best lower bound.
Unlike the study of $u(n)$, in recent years there have been several small improvements in the upper bound for $u_3(n)$.
The current best bound, by Zahl  \cite{Z}, states that for any $\eps>0$ we have $\displaystyle u_3(n)= O\left(n^{\frac{295}{197}+\epsilon}\right)$.

For a positive integer $k$, we denote as $\CaC^{(3)}_{k}(n)$ the maximum number of $k$-chains that can be spanned by $n$ points in $\RR^3$, for any sequence of distances $(\delta_1,\delta_2,\ldots,\delta_{k})$.
We only consider the case where $k$ is a constant that does not depend on $n$.
We trivially have $\CaC^{(3)}_{0}(n)=n$.
For any $\eps>0$, we have that $\CaC^{(3)}_1(n)=u_3(n)=O\left(n^{\frac{295}{197}+\epsilon}\right)$.
The following is our main result in $\RR^3$.

\begin{theorem} \label{th:main R3}
For every $k\ge 2$ and any $\eps>0$, we have
\[
\CaC^{(3)}_{k}(n) =
\begin{cases}
O\left(n^{2k/3+1}\right) \qquad & \text{if } \, k\equiv 0 \Mod{3}, \\
O\left(n^{2k/3 + 23/33 +\eps}\right) & \text{if } \, k\equiv 1 \Mod{3}, \\
O\left(n^{2k/3+2/3}\right) & \text{if } \, k\equiv 2 \Mod{3}.
\end{cases}
\]
\end{theorem}
\noindent In particular, note that $\CaC^{(3)}_{2}(n) = \CaC_{2}(n)=\Theta(n^2)$.

In the case of $\CaC^{(3)}_3(n)$, one can obtain a slight improvement over Theorem \ref{th:main R3} with a simple use of Zahl's bound $u_3(n)=O\left(n^{\frac{295}{197}+\epsilon}\right)$.
Given a sequence of distances $(\delta_1,\delta_2,\delta_3)$ and a point set $\pts$, we look for 3-chains $(p_1,p_2,p_3,p_4)\in \pts^4$.
There are at most $u_3(n)$ options for choosing the pair $(p_1,p_2)\in\pts^2$ and at most $u_3(n)$ options for choosing $(p_3,p_4)\in\pts^2$. 
This implies that $\CaC^{(3)}_3(n)=O\left((u_3(n))^2\right)=O\left(n^{2.995}\right)$.
The same approach does not yield improved bounds for any larger values of $k$. 

With respect to lower bounds, Erd\H os's unit distances bound implies $\CaC^{(3)}_{1}(n)  =u_3(n)=\Omega(n^{4/3}\log \log n)$.
For longer chains, we derive the following bounds. This construction was motivated by Oliver Purwin.

\begin{proposition}\label{pr:LowerR3}
For any integer $k\ge 2$, we have
\[
\CaC^{(3)}_k(n) =
\begin{cases}
\Omega\left(n^{(k+1)/2}\right) \qquad & \text{if } \, k \text{ is odd}, \\
\Omega\left(n^{k/2+1}\right) & \text{if } \, k \text{ is even}.
\end{cases}
\]
\end{proposition}

In $\RR^d$ with $d\ge 4$, the unit distances problem becomes significantly simpler.
In particular, in this case it is possible to have $\Theta(n^2)$ pairs of points at a distance of one (see \cite{Lenz55}).
The same construction immediately implies that one can have $\Theta(n^{k+1})$ chains of length $k$, for any $k\ge 1$.
Thus, the chains problem is trivial in this case.


In Section \ref{sec:Prelim} we introduce geometric incidence results that we require for our proofs.
In Section \ref{sec:R2Bounds} we derive our bounds in $\RR^2$.
Finally, in Section \ref{sec:R3Bounds} we derive our bounds in $\RR^3$.

\section{Preliminaries: Geometric incidences} \label{sec:Prelim}

Given a set $\pts$ of points and a set $\circs$ of circles, both in $\RR^2$, an \emph{incidence} is a pair $(p,\gamma)\in\pts \times \circs$ such that the point $p$ is contained in the circle $\gamma$.
We denote by $I(\pts,\circs)$ the number of incidences in $\pts \times \circs$.
Aronov and Sharir \cite{AS02} proved the following result.

\begin{theorem} \label{th:IncCirc}
Consider a set $\pts$ of $m$ points and a set $\circs$ of $n$ circles, both in $\RR^2$.
For every $\eps>0$, we have
\[ I(\pts,\circs) = O\left(m^{9/11+\eps}n^{6/11} + m^{2/3}n^{2/3}+m+ n\right).\]
This bound also holds when assuming that $\pts$ and $\circs$ lie on a common sphere in $\RR^3$.
\end{theorem}

There are many similar incidence problems where we have a set of points $\pts$ and a set of ``objects'' $\circs$ in $\RR^d$, and are looking for the maximum number of incidences.
In every such problem we use $I(\pts,\circs)$ to denote the number of incidences in $\pts \times\circs$.
The following result was proved by Zahl \cite{Zahl13} and independently by Kaplan, Matou\v sek,  Safernov\'a, and Sharir \cite{KMSS12}.

\begin{theorem} \label{th:IncSphere}
Consider a set $\pts$ of $m$ points and a set $\spheres$ of $n$ spheres of the same radii, both in $\RR^3$.
For every $\eps>0$, we have
\[ I(\pts,\spheres) = O\left(m^{3/4}n^{3/4} + m+ n\right).\]
\end{theorem}

Let $\spheres$ be a set of spheres of the same radii in $\RR^3$.
For an integer $r\ge 2$, we say that a point $p\in \RR^2$ is $r$-\emph{rich} if $p$ is incident to at least $r$ spheres of $\spheres$.
Upper bounds for incidence problems such as Theorems \ref{th:IncCirc} and \ref{th:IncSphere} have \emph{dual formulations} in terms of $r$-rich points.
These results are dual in the sense that there are short simple ways of getting from each upper bound to the other (for example, see \cite{ST83}).
The following is the dual formulation of Theorem \ref{th:IncSphere}.

\begin{theorem} \label{th:IncSphereDual}
Consider a set $\spheres$ of $n$ spheres of the same radii in $\RR^3$.
For every integer $r\ge 2$, the number of $r$-rich points is
\[ O\left(\frac{n^3}{r^4} +\frac{n}{r}\right).\]
\end{theorem}

The following is the dual form of the upper bound for the number of incidences with circles of the same radii in $\RR^2$.
In other words, it is the dual of the current best upper bound for $u(n)$.

\begin{theorem} \label{th:UnitCircDual}
Consider a set $\circs$ of $n$ circles of the same radii in $\RR^2$.
For every integer $r\ge 2$, the number of $r$-rich points is
\[ O\left(\frac{n^2}{r^3} +\frac{n}{r}\right).\]
\end{theorem}

\section{Bounds for the number of $k$-chains in $\RR^2$} \label{sec:R2Bounds}

In this section we prove our bounds for the maximum number of $k$-chains in $\RR^2$.
We refer to a 2-chain as a \emph{hinge}.
As a warm-up, we first derive a tight bound for the case of hinges.
We will also require this bound to prove Theorem \ref{th:mainR2}.

\begin{lemma}\label{le:hinges}
$\displaystyle \CaC_2(n)=\Theta\left(n^2\right)$.
\end{lemma}
\begin{proof}
Let $\pts$ be a set of $n$ points in $\RR^2$.
Consider a sequence of distances $(\delta_1,\delta_2)$, and denote a hinge as $(p_1,p_2,p_3)\in \pts^3$.
There are $n$ choices for $p_1$, and then at most $n-1$ choices for $p_3$.
After these two points are chosen, $p_2$ must be on a circle of radius $\delta_1$ centered at $p_1$ and also on a circle of radius $\delta_2$ centered at $p_2$.
These two circles intersect in at most two points, so for every choice of $p_1$ and $p_3$ there are at most two valid options for $p_2$.
This immediately implies that $\CaC_2(n) = O\left(n^2\right)$.

To see that the above upper bound is tight, consider a sequence of distances $(\delta_1,\delta_2)$.
Let $C_1$ and $C_2$ be circles centered at the origin of respective radii $\delta_1$ and $\delta_2$.
We construct a set $\pts$ by taking the origin, $\lfloor (n-1)/2 \rfloor$ points from $C_1$, and $\lceil (n-1)/2 \rceil$ points from $C_2$.
It is not difficult to verify that $\pts$ is a set of $n$ points that span $\Theta(n^2)$ hinges.
\end{proof}

We now move to derive upper bounds for $\CaC_k(n)$ when $k\ge 3$.
We first recall the statement of the Theorem \ref{th:mainR2}.
\vspace{2mm}

\noindent {\bf Theorem \ref{th:mainR2}.}
\emph{For every $k\ge 3$, we have}
\[ \CaC_{k}(n)= O\left( n^{2k/5+ 1+\gamma(k)}\right), \]
\emph{where}
\[
\gamma(k) =
\begin{cases}
\frac{1}{75}\cdot\left( 4 - 4\cdot (-1/4)^{k/4}  \right) \qquad & \text{ if } k \equiv 0 \Mod{4}, \\
\frac{1}{75}\cdot\left( 4 - 9\cdot (-1/4)^{\floor{k/4}}  \right) & \text{ if } k \equiv 1 \Mod{4}, \\
\frac{1}{75}\cdot\left( 4 + 11\cdot (-1/4)^{\floor{k/4}} \right) & \text{ if } k \equiv 2 \Mod{4}, \\
\frac{1}{75}\cdot\left( 4 - \frac{13}{2} \cdot (-1/4)^{\floor{k/4}}   \right) & \text{ if } k \equiv 3 \Mod{4}.
\end{cases}
\]

\begin{proof}
The proof consists of two parts.
In the first part we derive a recurrence relation for upper bounds on $\CaC_k(n)$.
In the second part we solve this relation.

\parag{Deriving a recurrence relation.}
Let $\pts$ be a set of $n$ points in $\RR^2$.
Consider a sequence of distances $(\delta_1,\ldots,\delta_k)$, and denote a $k$-chain as $(p_1,\ldots,p_{k+1})\in \pts^{k+1}$.
Let $Q\subset \pts^{k+1}$ denote the set of $k$-chains that correspond to the sequence of distances.

For a point $p_{2}\in \RR^2$, we denote by $\alpha_{1}(p_{2})$ the number of points $p_{1}\in \pts$ that satisfy $|p_{1}p_{2}|=\delta_{1}$.
Similarly, for a point $p_k\in \pts$ we denote by $\alpha_{k+1}(p_k)$ the number of points $p_{k+1}\in \pts$ that satisfy $|p_kp_{k+1}|=\delta_k$.
We partition $Q$ into two disjoint sets $Q'$ and $Q''$, as follows.
Let $Q'$ be the set of $k$-chains $(p_1,\ldots,p_{k+1})\in Q$ satisfying $\alpha_{1}(p_{2})\ge \alpha_{k+1}(p_k)$.
Let $Q'' = Q\setminus Q'$ be the set of $k$-chains $(p_1,\ldots,p_{k+1})\in Q$ satisfying $\alpha_{1}(p_{2})< \alpha_{k+1}(p_k)$.

Without loss of generality, assume that $|Q''|\geq |Q'|$.
In this case, we have that $|Q|\le 2|Q''|$.
For $0\le j < \lceil\log_2 n \rceil$, let $Q''_{j}$ be the set of $k$-chains $(p_1,\ldots,p_{k+1})\in Q''$ that satisfy $2^{j-1}\le \alpha_{k+1}(p_k) < 2^j$.
In other words, we dyadically decompose $Q''$ into $\lceil\log_2 n \rceil$ subsets, each consisting of $k$-chains $(p_1,\ldots,p_{k+1})$ having approximately the same value of $\alpha_{k+1}(p_k)$.
Note that $|Q''|= \sum_{j=0}^{\lceil\log_2 n\rceil} |Q''_{j}|$.

For a fixed $0\le j < \lceil\log_2 n\rceil$, we now derive an upper bound on the number of $k$-chains in $Q''_{j}$.
When placing a circle of radius $\delta_k$ around every point of $\pts$, a point $p_k$ that satisfies $\alpha_{k+1}(p_k) \ge 2^{j-1}$ is also a $2^{j-1}$-rich point (as defined in Section \ref{sec:Prelim}).
Thus, Theorem \ref{th:UnitCircDual} implies that the number of points $p_k$ that satisfy $\alpha_{k+1}(p_k)\ge 2^{j-1}$ is
\[ O\left(\frac{n^2}{2^{3j}}+\frac{n}{2^j}\right). \]
The above is an upper bound for the number of ways to choose $p_k$.
After choosing a specific $p_k$, there are fewer than $2^j$ choices for $p_{k+1}$.
There are at most $\CaC_{k-3}(n)$ choices for $(p_1,\ldots,p_{k-2})\in \pts^{k-2}$.
Finally, after choosing $(p_1,\ldots,p_{k-2})$ and $p_k$ there are at most two choices for $p_2$, since this point lies on the intersection of two circles (as in the proof of Lemma \ref{le:hinges}).
We conclude that
\begin{equation} \label{eq:kChainLargeJ}
|Q''_{j}| \le O\left(\frac{n^2}{2^{3j}}+\frac{n}{2^j}\right) \cdot 2^j \cdot \CaC_{k-3}(n) \cdot 2 = O\left(\frac{n^2\cdot \CaC_{k-3}(n)}{2^{2j}}+n\cdot \CaC_{k-3}(n)\right).
\end{equation}

The bound of \eqref{eq:kChainLargeJ} is reasonable for large values of $j$, but weak for small values of $j$.
We thus derive another bound for the number of $k$-chains in $Q''_{j}$.
There are at most $\CaC_{k-2}(n)$ choices of $(p_2,\ldots,p_k)\in \pts^{k-1}$.
As before, there are fewer than $2^j$ choices for $p_{k+1}$.
Since we are only counting quadruples from $Q''$, there are also fewer than $2^j$ choices for $p_1$.
Combining these observations leads to
\begin{equation} \label{eq:kChainSmallJ}
|Q''_{j}| \le \CaC_{k-2}(n) \cdot 2^{2j}.
\end{equation}

Set $\beta = \left\lfloor\log_2\left(\frac{n^2 \cdot \CaC_{k-3}(n)}{\CaC_{k-2}(n)}\right)^{1/4}\right\rfloor$.
By combining \eqref{eq:kChainLargeJ} and \eqref{eq:kChainSmallJ}, we obtain
\begin{align*}
|Q|\le 2|Q''| &= 2\sum_{j=0}^{\lg n} |Q''_{j}| = 2\sum_{j=0}^{\beta} |Q''_{j}| + 2\sum_{j=\beta+1}^{\lceil \log_2 n\rceil} |Q''_{j}| \\[2mm]
&= O\left(\sum_{j=0}^{\beta}\CaC_{k-2}(n) \cdot 2^{2j}+\sum_{j=\beta+1}^{\lceil \log_2 n\rceil} \left(\frac{n^2\cdot \CaC_{k-3}(n)}{2^{2j}}+n\cdot \CaC_{k-3}(n)\right)\right) \\[2mm]
&=O\left(n \cdot \sqrt{\CaC_{k-2}(n) \cdot \CaC_{k-3}(n)} + n \log n \cdot \CaC_{k-3}(n) \right).
\end{align*}

The first term in this bound dominates when $\CaC_{k-2}(n) =\Omega\left(\CaC_{k-3}(n)\log^2 n\right)$.
This will be the case for all of our bounds, so we may ignore the second term and consider the bound
\begin{equation} \label{eq:RecRel}
\CaC_k(n) = O\left(n \cdot \sqrt{\CaC_{k-2}(n) \cdot \CaC_{k-3}(n)} \right).
\end{equation}

\parag{Solving the recurrence relation.}
Recall that $\CaC_0(n)=n$ and that $\CaC_1(n)=u(n) = O(n^{4/3})$.
Combining these with \eqref{eq:RecRel} yields $\CaC_3(n) = O\left(n^{13/6}\right)$.
We can similarly derive an upper bound for $\CaC_k(n)$ with any $k\ge 3$.

Let $a_k$ to be the current best exponent of $n$ in the upper bound for $\mathcal C_{k}(n)$.
As already mentioned above, we have $a_0=1, a_1=4/3$, and $a_2 =2$.
By \eqref{eq:RecRel}, we have that
\[ a_k = \frac{1}{2}a_{k-3}+\frac{1}{2}a_{k-2}+1. \]

Solving this relation and initial values yields
$$a_k = \frac{79}{75} + \left(-\frac{2}{75} - \frac{11}{75}i \right) \left(-\frac{1}{2} - \frac{i}{2} \right)^k + \left(-\frac{2}{75} + \frac{11}{75}i \right) \left(-\frac{1}{2}+\frac{i}{2} \right)^k + \frac{2}{5}k .$$
Simplifying this leads to
\[
a_k =
\begin{cases}
1 + \frac{2}{5}k + \frac{1}{75}\left( 4 - 4\cdot (-1/4)^{k/4}  \right) \qquad & \text{ if } k \equiv 0 \Mod{4}, \\
1 + \frac{2}{5}k + \frac{1}{75}\left( 4 - 9\cdot (-1/4)^{\floor{k/4}}  \right) & \text{ if } k \equiv 1 \Mod{4}, \\
1 + \frac{2}{5}k + \frac{1}{75}\left( 4 + 11\cdot (-1/4)^{\floor{k/4}} \right) & \text{ if } k \equiv 2 \Mod{4}, \\
1 + \frac{2}{5}k + \frac{1}{75}\left( 4 - \frac{13}{2} \cdot (-1/4)^{\floor{k/4}}   \right) & \text{ if } k \equiv 3 \Mod{4}.
\end{cases}
\]
\end{proof}

We now derive our lower bounds for $\CaC_k(n)$.
We first recall the statement of the corresponding proposition.
\vspace{2mm}

\noindent {\bf Proposition \ref{pr:LowerR2}.}
\emph{For any integer $k\ge 0$, we have}
\[
\CaC_k(n) =
\begin{cases}
\Omega\left(n^{k/3+1}\right) \qquad & \text{if } \, k\equiv 0 \Mod{3}, \\
\Omega\left(n^{(k+2)/3}\right) & \text{if } \, k\equiv 1 \Mod{3}, \\
\Omega\left(n^{(k+1)/3+1}\right) & \text{if } \, k\equiv 2 \Mod{3}.
\end{cases}
\]

\begin{figure}[h]
\centering
\includegraphics[scale=0.64]{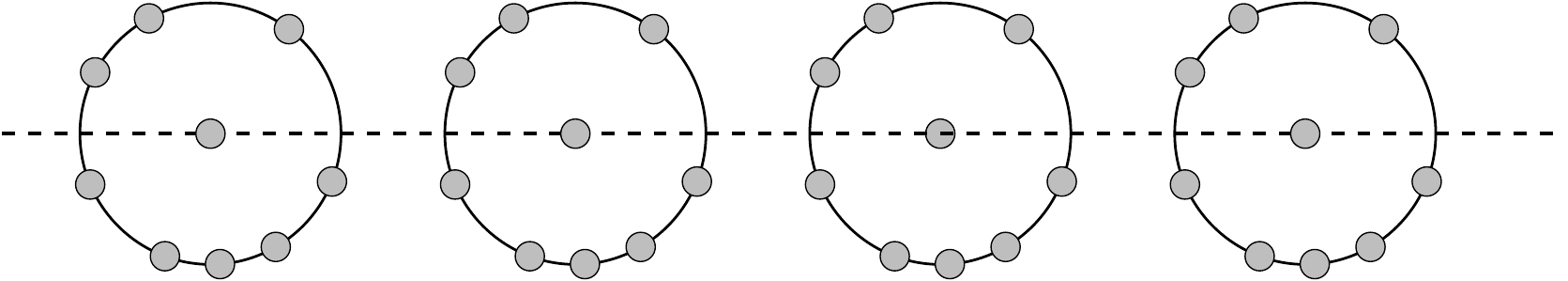}
\caption{$(k+1)/3$ translated copies of the same circle configuration.}
\label{fi:NewChain1}
\end{figure}

\begin{proof}
First assume that $k \equiv 2 \Mod{3}$.
We choose two arbitrary distances $\delta_1,\delta_2>0$, and consider the sequence of distances
\[ (\delta_1, \delta_1, \delta_2,\delta_1,\delta_1,\delta_2,\delta_1,\delta_1,\delta_2,\ldots,\delta_2,\delta_1,\delta_1). \]

Set $m = \lfloor3n/(k+1)\rfloor$.
Let $\gamma$ be a circle of radii $\delta_1$, place $m-1$ on $\gamma$, and place one additional point at the center of $\gamma$.
Denote this configuration of $m$ points as $A$.
We create a second copy of $A$, translated a distance of $\delta_2$ in the $x$ direction.
That is, we now have two circles and $2m$ points.
We keep creating more copies of $A$, each translated a distance of $\delta_3$ in the $x$-direction from the preceding one.
After having $(k+1)/3$ copies of $A$, we denote the resulting set of $m(k+1)/3\le n$ points as $\pts$.
Such a configuration is illustrated in Figure \ref{fi:NewChain1}.

To obtain a chain, we first choose $p_1$ to be a point on the first circle and set $p_2$ to be the center of that circle.
We have $m-1=\Theta(n)$ options for choosing $p_1$, and a single choice for $p_2$.
We then choose $p_3$ to be another point on the first circle, set $p_4$ to be the point in the second copy of $A$ that corresponds to $p_3$, and set $p_5$ to be the center of the second copy of $A$.
There are $m-2=\Theta(n)$ choices for $p_3$, a single option for $p_4$, and a single option for $p_5$.
We repeat this step $(k+1)/3-1$ times: Starting from a center of a circle, choosing a point on the circle, moving to the corresponding point in the next circle, and moving to the center of that next circle.
At each step we determine three vertices of the $k$-chain and have $m-2=\Theta(n)$ choices.

When the above process ends, we obtain a $k$-chain that corresponds to our sequence of distances.
This process consists of $(k+1)/3 +1$ steps where we have at least $m-2=\Theta(n)$ choices.
Thus, the number of $k$-chains that correspond to the above sequence of distances is $\Theta\left(n^{(k+1)/3 +1}\right)$.

Next, consider the case when $k\equiv 1 \Mod{3}$.
In this case, we repeat the above construction with $m = \lfloor3n/(k+2)\rfloor$ and have $(k+2)/3$ copies of $A$.
When creating a $k$-chain as before, we only take two points from the rightmost circle.
Thus, the sequence of distances ends with $(\ldots, \delta_2,\delta_1)$.
This implies that we have only $(k+2)/3$ steps with $\Theta(n)$ choices.
That is, the number of $k$-chains is $\Theta\left(n^{(k+2)/3}\right)$.

The case of $k\equiv 0 \Mod{3}$ is handled symmetrically.
We repeat the above construction with $m = \lfloor3n/(k+3)\rfloor$, and have $(k+3)/3$ circles.
When creating a $k$-chain, we only take one points from the rightmost circle.
Thus, the sequence of distances ends with $(\ldots, \delta_2)$.
This implies that we have only $(k+3)/3 = k/3+1$ steps with $\Theta(n)$ choices.
That is, the number of $k$-chains is $\Theta\left(n^{k/3+1}\right)$.
\end{proof}

We conclude this section with a proof of Proposition \ref{pr:OptimisticBound}.
We first recall the statement of this proposition.
\vspace{2mm}

\noindent {\bf Proposition \ref{pr:OptimisticBound}.}
\emph{For every $k\ge 3$, we have}
\[
\CaC_{k}(n) =
\begin{cases}
O\left(n \cdot u_2(n)^{k/3}\right) \qquad & \text{if } \, k\equiv 0 \Mod{3}, \\
O\left(u_2(n)^{(k+2)/3}\right) & \text{if } \, k\equiv 1 \Mod{3}, \\
O\left(n^2\cdot u_2(n)^{(k-2)/3}\right) & \text{if } \, k\equiv 2 \Mod{3}.
\end{cases}
\]

\begin{proof}
Let $\pts$ be a set of $n$ points in $\RR^2$.
Consider a sequence of distances $(\delta_1,\ldots,\delta_k)$, and denote a $k$-chain as $(p_1,\ldots,p_{k+1})\in \pts^{k+1}$.
The proof is split into cases according to $k \Mod{3}$.

First, we consider the case of $k\equiv 0 \Mod{3}$.
There are $n$ possible choices for $p_{k+1}$.
We split the first $k$ points into triples of the form $(p_{3j+1}, p_{3j+2}, p_{3j+3})$, for every $0\le j < k/3$.
In each of these $\frac{k}{3}$ triples, we have at most $u_2(n)$ choices for the pair $(p_{3j+1}, p_{3j+2})$.
After choosing the first pair of points in each triple, there remain $k/3$ points that were not yet chosen in the chain.
For every $0\le j < k/3$, since we already chose both $p_{3j+2}$ and $p_{3j+4}$, there are at most two valid choices for $p_{3j+3}$ (using the same argument as in the proof of Lemma \ref{le:hinges}).
We conclude that
\[ \CaC_{k}(n)=O\left( n\cdot u_2(n)^{k/3}\right). \]

We move to consider the case of $k\equiv 1 \Mod{3}$.
There are at most $u_2(n)$ choices for the pair of points $(p_k,p_{k+1})\in\pts^2$.
We split the first $k-1$ points into triples of the form $(p_{3j+1}, p_{3j+2}, p_{3j+3})$, for every $0\le j < (k-1)/3$.
By handling these $(k-1)/3$ triples as in the previous case, we obtain
\[ \CaC_{k}(n)= O\left( u(n)^{(k+2)/3} \right). \]

Finally, we consider the case of $k\equiv 2 \Mod{3}$.
Since $(p_{k-1},p_k,p_{k+1})$ form a hinge, by Lemma \ref{le:hinges} there are $O(n^2)$ choices for this triple.
We split the remaining $k-2$ points into triples of the form $(p_{3j+1}, p_{3j+2}, p_{3j+3})$, for every $0\le j < (k-2)/3$.
By handling these $(k-2)/3$ triples as in the previous cases, we obtain
\[ \CaC_{k}(n)=O\left(n^2\cdot u(n)^{(k-2)/3}\right). \]
\end{proof}



\section{Bounds for the number of $k$-chains in $\RR^3$} \label{sec:R3Bounds}

In this section we derive our bounds for $\CaC^{(3)}_k(n)$.
We begin by recalling the statement of Theorem \ref{th:main R3}.
\vspace{2mm}

\noindent {\bf Theorem \ref{th:main R3}.}
\emph{For every $k\ge 2$ and any $\eps>0$, we have}
\[
\CaC^{(3)}_{k}(n) =
\begin{cases}
O\left(n^{2k/3+1}\right) \qquad & \text{if } \, k\equiv 0 \Mod{3}, \\
O\left(n^{2k/3 + 23/33 +\eps}\right) & \text{if } \, k\equiv 1 \Mod{3}, \\
O\left(n^{2k/3+2/3}\right) & \text{if } \, k\equiv 2 \Mod{3}.
\end{cases}
\]
Moreover, $\CaC^{(3)}_3(n)=O((u_3(n))^2).$

\begin{proof}
The proof has two steps.
We first prove that $\CaC^{(3)}_{2}(n) = \Theta(n^2)$, and then rely on this bound to study large values of $k$.
Recall that we refer to 2-chains as hinges.

\parag{Deriving a bound for the number of hinges and 3-chains.}
The lower bound $\CaC^{(3)}_{2}(n) = \Omega(n^2)$ is obtained in the same way as in Lemma \ref{le:hinges}.
It remains to derive a matching upper bound.

Let $\pts$ be a set of $n$ points in $\RR^3$.
Consider a sequence of distances $(\delta_1,\delta_2)$, and denote the hinges corresponding to this sequence as $(p_1,p_2,p_3)\in \pts^3$.
For an integer $0\le j < \lceil\log_2 n \rceil$, let $\pts_j$ be the set of points $p_2\in \pts$ such that there exist at least $2^{j-1}$ points $p_1\in \pts$ satisfying $|p_1p_2|=\delta_1$, and less than $2^j$ such points.
When placing a sphere of radius $\delta_1$ around every point of $\pts$, every point $p_2\in \pts_j$ is also a $2^{j-1}$-rich point.
Thus, by Theorem \ref{th:IncSphereDual}
\[ |\pts_j|=O\left(\frac{n^3}{2^{4j}}+ \frac{n}{2^j}\right). \]

For a fixed $0\le j < \lceil\log_2 n \rceil$, we now bound the number of hinges with $p_2\in \pts_j$.
We know that for every $p_2\in \pts_j$, there are $\Theta(2^j)$ choices for $p_1$.
Let $\spheres$ be a set of $|\pts_j|$ spheres of radius $\delta_2$ centered around the points of $\pts_j$.
Every incidence in $\pts_j\times \spheres$ corresponds to a pair $(p_2,p_3)\in \pts_j\times \pts$ that appear together in a hinge.
By Theorem \ref{th:IncSphere}, the number of hinges with $p_2\in \pts_j$ is
\begin{align*}
O\left(2^j \right) \cdot I(\pts,\spheres) &= O\left(2^j \cdot \left(n^{3/4}|\pts_j|^{3/4} +n + |\pts_j|\right) \right) \\[2mm]
&=O\left(2^j \cdot \left(n^{3/4}\left(\frac{n^3}{2^{4j}}+ \frac{n}{2^j}\right)^{3/4} + n + \left(\frac{n^3}{2^{4j}}+ \frac{n}{2^j}\right) \right)\right) \\[2mm]
&=O\left(\frac{n^3}{2^{2j}} + n^{3/2}2^{j/4} + n\cdot 2^j \right).
\end{align*}

Summing the above bound over every $j$ from $\lceil\log_2 n \rceil/2$ up $\lceil\log_2 n \rceil-1$ leads to $O(n^2)$.
That is, the number of hinges with $p_2\in \pts_j$ for any $\lceil\log_2 n \rceil/2 \le j \le \lceil\log_2 n \rceil$ is $O(n^2)$.
It remains to bound the number of hinges with $p_2\in \pts_j$ for any $0\le j <\lceil\log_2 n \rceil/2$.

By definition, we also have $|\pts_j|\le |\pts|=n$.
Repeating the above argument with this bound implies that the number of hinges with $p_2\in \pts_j$ is
\[ O\left(2^j \right) \cdot I(\pts,\spheres) = O\left(2^j \cdot \left(n^{3/4}|\pts_j|^{3/4} +n + |\pts_j|\right) \right)  = O\left(2^jn^{3/2}\right).\]

Summing the above bound over every $j$ from $0$ up to $\lceil\log_2 n \rceil/2-1$ leads to $O(n^2)$.
We conclude that the total number of hinges in $\pts$ is $O(n^2)$.

To bound the number of 3-chains, we simply notice that there are no more than $u_3(n)$ ways to choose $p_1$ and $p_2$ such that $|p_1p_2|=\delta_1,$ and no more than $u_3(n)$ ways to choose $p_3$ and $p_4$ such that $|p_3p_4|=\delta_3.$

\parag{Deriving bounds for larger values of $k$.}
Let $\pts$ be a set of $n$ points in $\RR^3$.
Consider a sequence of distances $(\delta_1,\ldots,\delta_k)$ and denote a $k$-chain as $(p_1,\dots,p_{k+1})\in \pts^{k+1}$.

First assume that $k \equiv 2 \Mod{3}$.
For $0\le j < (k+1)/3$, by the above bound for $\CaC^{(3)}_2(n)$ there are $O(n^2)$ ways for choosing $(p_{3j+1},p_{3j+2},p_{3j+3})$.
Thus, in this case the number of $k$-chains is $O(n^{2(k+1)/3})$.
Next, we assume that $k \equiv 0 \Mod{3}$.
For $0\le j < k/3$, there are $O(n^2)$ ways for choosing $(p_{3j+1},p_{3j+2},p_{3j+3})$.
Since there are $n$ ways for choosing $p_{k+1}$, we get a total of $O(n^{2k/3+1})$ chains.

Finally, consider the case of $k \equiv 1 \Mod{3}$.
For $0\le j < (k-1)/3$, there are $O(n^2)$ ways for choosing $(p_{3j+1},p_{3j+2},p_{3j+3})$.
Consider such a fixed choice of $p_1,\ldots,p_{k-1}$.
Let $S$ be the sphere centered at $p_{k-1}$ and of radius $\delta_{k-1}$.
There are $n$ ways for choosing $p_{k+1}$, and for every such choice $p_k$ must be on a specific circle on $S$.
In particular, this circle is the intersection of  $S$ with the sphere centered at $p_{k+1}$ and of radius $\delta_k$.
Let $\circs$ denote the set of circles that are obtained in this way.

A specific circle in $\circs$ can originate from at most two values of $p_{k+1}$, since at most two spheres of radius $\delta_k$ can contain a given circle.
This implies that $|\circs| = \Theta(n)$.
For a fixed $p_{k+1}$, the number of choices for $p_k$ is the number of points on the corresponding circle.
Thus, the total number of choices for both $p_{k+1}$ and $p_k$ is $I(\pts,\circs)$.
Theorem \ref{th:IncCirc} implies that $I(\pts,\circs) = O(n^{15/11+\eps})$ for any $\eps>0$.
We conclude that, in the case of $k \equiv 1 \Mod{3}$, the number of $k$-chains is
\[ O\left(n^{2(k-1)/3}\cdot n^{15/11+\eps}\right) = O\left(n^{2k/3 + 23/33 +\eps}\right). \]
\end{proof}

We now prove Proposition \ref{pr:LowerR3}.
We begin by recalling the statement of this proposition.
\vspace{2mm}

\noindent {\bf Proposition \ref{pr:LowerR3}.}
\emph{For any integer $k\ge 2$, we have}
\[
\CaC^{(3)}_k(n) =
\begin{cases}
\Omega\left(n^{(k+1)/2}\right) \qquad & \text{if } \, k \text{ is odd}, \\
\Omega\left(n^{k/2+1}\right) & \text{if } \, k \text{ is even}.
\end{cases}
\]

\begin{proof}
The proof is a variant of the proof of Proposition \ref{pr:LowerR2}, taking advantage of the extra dimension that is available in this case.

First assume that $k$ is even.
We choose two arbitrary distances $0< \delta_1<\delta_2$, and consider the sequence of distances
\[ (\delta_1, \delta_2,\delta_1,\delta_2,\delta_1,\delta_2,\ldots,\delta_2,\delta_1). \]

Set $m = \lfloor2n/k\rfloor$.
Let $\gamma$ be a circle of radius $\delta_1$ centered at the origin of $\RR^3$ and contained in the plane defined by $x=0$.
We place $m-1$ points on $\gamma$, and one additional point at the center of $\gamma$.
Denote this configuration of $m$ points as $A$.
We create a second copy of $A$, translated in the $x$ direction such that the distance between the origin and every point on the translated copy of $\gamma$ is $\delta_2$.
We now have two circles and $2m$ points.
We keep creating more copies of $A$, each translated the same distance in the $x$-direction from the preceding one.
After having $k/2$ copies of $A$, we denote the resulting set of $mk/2\le n$ points as $\pts$.

To obtain a chain, we first choose $p_1$ to be a point on the first circle and set $p_2$ to be the center of that circle.
We have $m-1=\Theta(n)$ options for choosing $p_1$, and a single way to choose $p_2$.
We then choose $p_3$ to be a point on the second circle, and $p_4$ to be the center of the second circle.
There are $m-1=\Theta(n)$ options for choosing $p_3$ and a single option for $p_4$.
We repeat this step another $k/2-2$ times: Starting from a center of a circle, choosing a point on the next circle, and moving to the center of that circle.
At each step we determine two vertices of the $k$-chain and have $m-1=\Theta(n)$ choices.

When the above process ends, we obtain a $k$-chain that corresponds to our sequence of distances.
This process consists of $k/2+1$ steps where we have $m-1=\Theta(n)$ choices.
Thus, the number of $k$-chains that correspond to the above sequence of distances is $\Theta\left(n^{k/2+1}\right)$.

Next, consider the case when $k$ is odd.
In this case, we repeat the above construction with $m = \lfloor2n/(k-1)\rfloor$ and have $(k-1)/2$ circles. 
Then we add one more point $p_{k+1}$ as the center of yet another circle, but with no points on the circle around it. 
When creating a $k$-chain as before, we always end with this fixed final point $p_{k+1}$.
Thus, the sequence of distances ends with $\delta_2$ in this case.
This implies that we have $(k+1)/2$ steps with $\Theta(n)$ choices.
That is, the number of $k$-chains is $\Theta\left(n^{(k+1)/2}\right)$.
\end{proof}

\end{document}